\documentclass{amsart}
\usepackage{amsmath, amssymb, amscd, amsthm, amsfonts, amstext,
  amsbsy}
\usepackage{mathrsfs}
\usepackage{enumerate}

\newcommand{\RR}{\ensuremath{\mathbb{R}}}
\newcommand{\CC}{\ensuremath{\mathbb{C}}}
\def\leng{{\rm lh}}
\def\dom{{\rm dom}}
\def\diam{{\rm diam}}
\newcommand{\pow}{\ensuremath{\mathscr{P}}}
\newcommand{\Bai}{\ensuremath{{}^\omega \omega}}
\newcommand{\seqo}{\ensuremath{{}^{<\omega} \omega}}
\newcommand{\Can}{\ensuremath{{}^\omega 2}}
\newcommand{\AD}{\ensuremath{{\rm \mathsf{AD}}}}
\newcommand{\ZF}{\ensuremath{{\rm \mathsf{ZF}}}}
\newcommand{\AC}{\ensuremath{{\rm \mathsf{AC}}}}
\newcommand{\ACOR}{\ensuremath{{\rm \mathsf{AC}_\omega(\mathbb{R})}}}
\newcommand{\conc}{{}^\smallfrown}
\newcommand{\imp}{\Rightarrow}
\newcommand{\fhi}{\varphi}
\newcommand{\bSigma}{\mathbf{\Sigma}}
\newcommand{\bPi}{\mathbf{\Pi}}
\newcommand{\bGamma}{\mathbf{\Gamma}}
\newcommand{\bDelta}{\mathbf{\Delta}}
\newcommand{\bN}{\mathbf{N}}
\newcommand{\F}{\mathcal{F}}
\newcommand{\G}{\mathcal{G}}
\newcommand{\B}{\mathcal{B}}
\newcommand{\D}{\mathcal{D}}
\newcommand{\Lip}{{\rm Lip}}
\newcommand{\restr}[2]{#1 \restriction  #2}
\newcommand{\seq}[2]{\langle #1 \mid #2 \rangle}

\newtheorem{theorem}{Theorem}[section]

\newtheorem{corollary}[theorem]{Corollary}
\newtheorem{proposition}[theorem]{Proposition}
\theoremstyle{definition}
\newtheorem{claim}{Claim}[theorem]
\newtheorem{remark}[theorem]{Remark}
\newtheorem{defin}{Definition}
\newtheorem{fact}{Fact}

\newcommand{\fullref}[1]{\ref{#1} on page\ \pageref{#1}}

\title{A NEW CHARACTERIZATION OF BAIRE CLASS $1$ FUNCTIONS}
\author{Luca Motto Ros}
\date{\today}
\address{Kurt G\"odel Research Center fo Mathematical Logic\\ University of Vienna \\ Austria}
\email{luca.mottoros@libero.it}
\thanks{Research partially supported by FWF Grant P 19898-N18.}

\begin{document}

\maketitle

\begin{abstract}
We give a new characterization of the
Baire class $1$ functions (defined on an ultrametric space) by proving
that they are exactly the pointwise limits of sequences of full
functions (which
are particularly simple Lipschitz functions). Moreover we highlight
the link between the two classical stratifications of the
Borel functions by showing that the Baire class functions of some level
are exactly those
obtained as uniform limits of  sequences of Delta functions (of a
corresponding level).
\end{abstract}

\section{Introduction}

If $X$ and $Y$ are metrizable spaces, a function $f: X \to  Y$
is said to be \emph{continuous} if the preimage of an open set of $Y$
is open with respect to the topology of $X$, i.e.\ if $f^{-1}(U) \in
\bSigma^0_1(X)$ for every $U \in \bSigma^0_1(Y)$. There are two
natural generalizations of this definition, namely functions such that
$f^{-1}(U) \in \bSigma^0_{\xi+1}(X)$ for every $U$ open in $Y$ and
functions such that $f^{-1}(S) \in \bSigma^0_\xi(X)$ for every $S \in
\bSigma^0_\xi(Y)$ (for $\xi < \omega_1$): the former are called
\emph{Baire class functions (of level $\xi$)} while the latter are
called \emph{Delta
  functions (of level $\xi$)}. Each generalization provides
a stratification of the Borel functions from $X$ to $Y$, but if we
compare the levels of the two hierarchies, that is if we fix some
$\xi<\omega_1$ in the definitions above, they are quite different: for
example, each level of the Delta functions is closed under
composition, while no level of the Baire class functions (apart from
continuous functions) has such a
property.

The Baire class stratification was introduced by Baire in 1899 (with a
slightly different definition which, however, turns out to be equivalent
to the one
proposed here in the relevant cases) and has been
extensively studied. Of particular interest are the Baire class $1$
functions, i.e.\ those functions such that the preimage of an open set
is a $\bSigma^0_2$ set. For example, if $f:[0,1] \rightarrow \RR$ is
differentiable (at endpoints we take one-side derivatives), then its
derivative $f'$ is of Baire class $1$. Moreover, Baire class $1$
functions (in particular those from the Baire space $\Bai$ or from any
compact space $X$ to $\RR$)
have lots of applications in the theory of Banach spaces (for more on this
subject see, for example, \cite{hor}, \cite{keclou}, \cite{kiria},
\cite{solecki}, \cite{rosenthal} and references quoted there).

In this paper we  will give a new
characterization for the
Baire class $1$ functions defined from an ultrametric  space $X$
(such as the Baire space $\Bai$ or the Cantor space $\Can$) to any
separable metric space $Y$, by showing that they are exactly the
pointwise limits of sequences of full functions (which are particular
Lipschitz functions) between $X$ and
$Y$. Moreover we will show that the two hierarchies presented before are
intimately related by proving that a function is of level $\xi$ in the
Baire class stratification just in case it is the uniform limit of
functions of level $\xi+1$ in the Delta stratification. In particular,
this gives another characterization of the Baire class $1$
functions (taking $\xi=1$).\\

The paper is organized as follows. In Section
\ref{sectionpreliminaries} we give some (old and new) definitions and
state the main Theorems of the paper. In Section \ref{sectionlink} we
consider the relations between Baire class and Delta functions, while
in Section \ref{sectionzero} we prove some Theorems about
zero-dimensional and ultrametric spaces. The results of these two
Sections are partially implicit in some classical proofs, but we put them
here since we want to highlight the link between the two
stratifications of the Borel-functions and the special properties of
Borel-partitions of completely disconnected spaces. Finally, in Section
\ref{sectionmain} we give the proof of the new characterization of the
Baire class $1$ functions.

All the proofs need only a very small fragment of the Axiom of Choice,
namely Countable Choice over the Reals ($\ACOR$ for
short)\footnote{The fact that we
  will not use the full Axiom of Choice becomes relevant if one wants
  to assume other axioms which contradict $\AC$ (which however are, in
  general, consistent with $\ACOR$). For example, the Axiom of
  Determinacy $\AD$ is needed to carry out the
  Wadge's analysis of continuous reducibility, so it could be useful to
  check that our results hold also in that context.}.
It seems
not possible to avoid this (very weak) assumption since it is needed
even to prove very basic results in Descriptive Set Theory, e.g.\ to
prove that $\bSigma^0_2(\RR)$ is closed under countable unions. Hence
we will always work under $\ZF+\ACOR$. All the metrics $d$ considered
throughout the paper are always assumed to be such that $d \leq 1$. This
condition is needed for the proofs of some of the results, but it is not
a true limitation. In fact, given any metric $d$ on $X$, it is easy to see
that $d' = \frac{d}{1+d}$ is a metric on $X$ compatible with $d$ such that $d'
\leq 1$. Moreover, $d$ is an ultrametric if and only if $d'$ is an ultrametric,
and one can easily check that all the definitions given in this paper are
``invariant'' under such a transformation of the metric, e.g.\ $A \subseteq
X$ is a full set with respect to $d$ just in case it is a full set with respect
to $d'$ (although with different constants\footnote{In particular, one constant
can be obtained from the other one via the bijection $j : \RR^+ \to (0,1) :
r \mapsto \frac{r}{1+r}$.}). Thus all the results hold also when considering
\emph{arbitrary} (ultra)metrics.
Finally, given any two sets $A$ and $B$,
we will denote by ${}^A B$ the set of all the functions from $A$ to
$B$ and by ${}^{<\omega} A$ the set of all the \emph{finite} sequences
of elements from $A$. In particular, $\Bai$ (the set of all the
$\omega$-sequences of
natural numbers) will denote the Baire space (endowed with the usual
topology), while ${}^{<\omega}\omega$ will denote the set of all the
\emph{finite} sequences of natural numbers. For all the other
undefined concepts and symbols we will always refer the
reader to
the standard monograph \cite{kechris}.\\

Finally, it is the author's pleasure to aknowledge his debt to S{\l}awomir Solecki
for his review of the present work and for the suggestion of a
further generalization of the characterization previously obtained.

\section{Preliminaries and statement of the main results}
\label{sectionpreliminaries}

We start with a few of definitions and basic results, following
closely the presentation of
\cite{kechris}.

\begin{defin}
Let $X$, $Y$ be metrizable spaces and $\xi < \omega_1$ a nonzero
ordinal. A function $f:X \rightarrow Y$ is of \emph{Baire class
$1$} if $f^{-1}(U) \in \bSigma^0_2(X)$ for every open set $U
\subseteq Y$. Recursively, for $1< \xi < \omega_1$ we define now a
function $f:X \rightarrow Y$ to be of \emph{Baire class $\xi$} if
it is the pointwise limit of a sequence of functions $f_n:X
\rightarrow Y$, where $f_n$ is of Baire class $\xi_n<\xi$.

We denote by $\B_\xi(X,Y)$ the sef of Baire class $\xi$ functions
from $X$ into $Y$.
\end{defin}

A function $f$ which is of Baire class $\xi$ (for some nonzero
countable ordinal $\xi$) is called a \emph{Baire class function}.

\begin{defin}
Let $X$, $Y$ be metrizable spaces and let $\Gamma$ be some collection
of subsets of $X$. We say that $f:X \rightarrow Y$ is
\emph{$\Gamma$-measurable} if $f^{-1}(U) \in \Gamma$ for every
open set $U \subseteq Y$.
\end{defin}

The link between $\Gamma$-measurable and Baire class function is
given by the following classical Theorem.

\begin{theorem}[Lebesgue, Hausdorff, Banach]\label{theorbairemeasurable}
Let $X$, $Y$ be metrizable spaces, with $Y$ separable. Then for $1
\leq \xi <\omega_1$, $f:X \rightarrow Y$ is of Baire class $\xi$
if and only if $f$ is $\bSigma^0_{\xi+1}$-measurable.
\end{theorem}

By analogy with respect to this Theorem, we say that a function $f$
between two metrizable spaces is of \emph{Baire class 0} if and
only if it is $\bSigma^0_1$-measurable, i.e.\ if and only if $f$
is continuous.\\

As a consequence of this Theorem, if $X$ and $Y$ are metrizable
spaces and $Y$ is separable, then the Baire class $\xi$ functions
provide a stratification in $\omega_1$ levels of all the Borel
functions, i.e.\ functions such that $f^{-1}(U)$ is Borel for any
$U \in \bSigma^0_1(Y)$ (Borel-measurable functions). In
fact for every nonzero countable $\xi$ and every $f \in
\B_\xi(X,Y)$, $f$ is clearly Borel. Conversely, let $U_n$ be a
countable basis for the topology of $Y$ and let $f$ be Borel. Let
$\mu_n$ be nonzero countable ordinals such that $f^{-1}(U_n) \in
\bSigma^0_{\mu_n}$ and let $\xi = \sup\{\mu_n \mid n \in \omega \}$
(which is again a nonzero countable ordinal). Since
$\bSigma^0_\xi$ is closed under countable unions and  $f^{-1}(U_n) \in
\bSigma^0_\xi$ for every $n
\in \omega$, we have that $f
\in \B_\xi(X,Y)$. Note also that any $\B_\xi(X,X)$ is not closed
under composition since, in general, if $f \in \B_\mu(X,Y)$ and $g
\in \B_\nu(Y,Z)$ then $g \circ f \in \B_{\mu+\nu}(X,Z)$. This
result follows from the fact that if $A \in \bSigma^0_{\nu}(Y)$
and $f \in \B_\mu(X,Y)$ then $f^{-1}(A) \in \bSigma^0_{\mu+\nu}$.

The following is another classical fact\footnote{In general, if $X$
  and $Y$ are metrizable with $Y$ separable and $f: X \rightarrow
  Y$ is the pointwise limit of a sequence of continuous functions then
  $f$ is of Baire class $1$. Nevertheless the converse fails in the
general case: for
a counterexample, simply take $X = \RR$ and $Y = \{0,1\}$ (with the
discrete metric)
and consider the function such that $f(0)=1$ and $f(x)=0$ for every $x
\neq 0$.}\label{note}.

\begin{theorem}[Lebesgue, Hausdorff, Banach]
\label{theorclassical} Let $X$, $Y$ be separable metrizable.
Moreover, assume that either $X$ is zero-dimensional or $Y = \RR^n$ for
some $n \in \omega$ (or even $Y=\CC^m$ or $Y = [0,1]^m$ for some
$m \in \omega$). Then $f:X \rightarrow Y$ is of Baire class $1$ if
and only if $f$ is the pointwise limit of a sequence of continuous
functions.
\end{theorem}

Hence, under the hypotheses of this Theorem, $f \in \B_\xi(X,Y)$
if and only if it is the pointwise limit of a sequence of
functions in $\bigcup_{\nu<\xi}\B_\nu(X,Y)$, for all $\xi \geq 1$.\\

There is another stratification of the Borel functions (in the case
$Y$ separable) which is
important because, contrary to the case of Baire class functions,
every level is a set of functions closed under composition.

\begin{defin}
Let $X$, $Y$ be metrizable spaces and $\xi < \omega_1$ be a
nonzero ordinal. A function $f: X \rightarrow Y$ is a
\emph{$\bDelta^0_\xi$-function} ($\bDelta^0_\xi$ for short) if
$f^{-1}(A) \in \bSigma^0_\xi(X)$ for every $A \in
\bSigma^0_\xi(Y)$.

We denote by $\D_\xi(X,Y)$ the set of such functions.
\end{defin}

\begin{proposition} \label{propositiondelta02functions}
For $\xi > 1$ the following are equivalent\footnote{For $\xi = 1$ we
  have in general that $i) \iff ii) \imp iii) \iff iv) \iff v)$ but
  not $iii) \imp i)$ (in fact if $Y$ is connected we have that
  \emph{every} function $f$ satisfies $iii)$, but $f$ is a
  $\bDelta^0_1$-function if and
  only if $f$ is continuous). Nevertheless the Proposition remains
  true even for $\xi = 1$ if we require that $Y$ is zero-dimensional.}:
\begin{enumerate}[i)]
\item $f$ is $\bDelta^0_\xi$;
\item $f^{-1}(A) \in
  \bPi^0_\xi$ for every $A \in \bPi^0_\xi$;
\item $f^{-1}(A) \in
  \bDelta^0_\xi$ for every $A \in \bDelta^0_\xi$;
\item $f^{-1}(A) \in
  \bSigma^0_\xi$ for every $\nu < \xi$ and $A \in \bPi^0_\nu$;
\item $f^{-1}(A) \in
  \bDelta^0_\xi$ for every $\nu < \xi$ and $A \in \bSigma^0_\nu$.
\end{enumerate}
\end{proposition}

\begin{proof}
Since $\bSigma^0_\xi$ is closed under countable union, it is easy
to see that $i) \iff iii)$. $i) \iff ii)$ is obvious, and also $iii)
\imp v)$ is trivial (since $\bSigma^0_\nu \subseteq
\bDelta^0_\xi$ for every $\nu < \xi$). $v) \imp iv)$ since
$\bDelta^0_\xi$ is closed under complementation and is contained by
definition in $\bSigma^0_\xi$. Finally, to see that $iv) \imp i)$
recall that, by definition, every
$\bSigma^0_\xi$ set $A$ can be written as a countable union of
$\bigcup_{\nu<\xi}\bPi^0_\nu$ sets.
\end{proof}

As in the case of Baire class functions, a function $f$ which is a
$\bDelta^0_\xi$-functions (for some nonzero
countable ordinal $\xi$) is called a \emph{Delta function}.

To observe that the Delta functions provide a stratification in
$\omega_1$ levels of all the Borel functions it is enough to observe
that every open set of $Y$ is  in $\bSigma^0_\xi(Y)$
for every nonzero countable
ordinal $\xi$ and every metrizable space $Y$ (and hence every Delta function is Borel) and that every
Baire class function is a Delta function. To see this, let $f \in
\B_\nu(X,Y)$ and let $\xi$ be the first additively closed ordinal
above $\nu$ (that is $\xi = \nu \cdot \omega$): we claim that $f$ is a
$\bDelta^0_\xi$-function. In fact, let $S \in \bSigma^0_\xi$: by
definition, $S = \bigcup_n P_n$, where each $P_n \in \bPi^0_{\mu_n}(Y)$
for some $\mu_n < \xi$. Since $f \in \B_\nu(X,Y)$ we have that
$Q_n = f^{-1}(P_n) \in \bPi_{\nu+\mu_n}(X)$ and hence $f^{-1}(S) =
\bigcup_n Q_n$ where each $Q_n$ is in $\bPi^0_{\nu+\mu_n}(X)$. Since
$\xi$ is additively closed and $\nu,\mu_n < \xi$ we have that
$\nu+\mu_n < \xi$ for every $n \in \omega$: therefore $f^{-1}(S) \in
\bSigma^0_\xi(X)$ by definition.

Moreover, using again the fact that $\bSigma^0_1(Y) \subseteq
\bSigma^0_\xi(Y)$, it is easy to check that $\D_{\xi+1}(X,Y)
\subseteq \B_\xi(X,Y)$.

\begin{defin}
Let $X$ and $Y$ be two metrizable spaces and let $\F \subseteq \G$
be two sets of functions from $X$ to $Y$. Then $\F$ is a
\emph{basis} for $\G$ just in case every function in $\G$  is the
\emph{uniform} limit of a sequence of functions in $\F$.
\end{defin}

We will prove in Section \ref{sectionlink} that each level of the
Delta functions
forms a basis for a corresponding level of the Baire class
functions. This result is essentially implicit in the proof of Theorem
\ref{theorbairemeasurable} (see \cite{kechris}), but we will reprove
it here for the sake of completeness.

\begin{theorem} \label{theorDelta0xi}
Let $(X,d_X)$, $(Y,d_Y)$ be two metric spaces and assume that $Y$
is also separable. A function $f:X \rightarrow Y$ is of Baire
class $\xi$ if and only if it is the \emph{uniform} limit  of a
sequence of $\bDelta^0_{\xi+1}$-functions.
\end{theorem}

\begin{corollary}
Let $(X,d_X)$, $(Y,d_Y)$ be two metric spaces and assume that $Y$
is also separable. A function $f:X \rightarrow Y$ is in $\B_1(X,Y)$ if
and only if it is the \emph{uniform} limit  of a sequence of
$\bDelta^0_2$-functions.
\end{corollary}

From  Theorem  \ref{theorDelta0xi} we can also derive the
following Corollary. It can be seen as an extension of Theorem
\ref{theorclassical}: in that case it was proved (under stronger
hypotheses) that $f$ is of Baire class $1$ if and only if it is
the pointwise limit of a sequence of $\bDelta^0_1$-functions
(i.e.\ continuous functions). Here we prove the same result for
every level different from $1$ (under weaker hypotheses).

\begin{corollary} \label{corDelta0xi}
Let $X$, $Y$ be two metrizable spaces and assume that $Y$ is also
separable. Then for every  $1<\xi<\omega_1$, $f:X \rightarrow Y$
is of Baire class $\xi$ if and only if $f$ is the pointwise limit
of a sequence of $\bDelta^0_{\xi}$-functions.
\end{corollary}

By Theorem \ref{theorclassical}, as previously observed, Corollary
\ref{corDelta0xi} remains true in the case $\xi=1$ if we require
that $X$ is separable and either $X$ is  zero-dimensional or $Y$
is one of $\RR^n$, $[0,1]^n$ or $\CC^n$ (for some $n \in \omega$).\\

Finally, we want to give a new characterization of the Baire class $1$
functions. First recall the following Definition.

\begin{defin}
Let $(X,d_X)$, $(Y,d_Y)$ be two metric spaces. A function $f:X
\rightarrow Y$ is \emph{Lipschitz (with constant $L \in \RR^+$)} if \[ \forall
x,x' \in X (d_Y(f(x),f(x')) \leq L\cdot d_X(x,x')).\]

We denote by $\Lip(X,Y;L)$ the set of such functions and put $\Lip(X,Y) =
\bigcup_{L \in \RR^+}\Lip(X,Y;L)$.
\end{defin}

Let now $(X,d_X)$ be an
\emph{ultrametric space}, i.e.\ a  metric space such that $d_X$ is an
ultrametric. A set $A \subseteq X$
is \emph{full (with constant $r \in \RR^+$)} if
\[ \forall x \in A (B(x,r) \subseteq A),\]
were $B(x,r) = \{ y \in X \mid d_X(x,y) < r \}$ is the usual open
ball.

\begin{proposition}\label{propfullbasicproperties}
  Let $(X,d_X)$ be an ultrametric space. Then the full subsets of $X$
  form an algebra. Moreover, an arbitrary union of balls with a fixed
  radius is full (in particular, an arbitrary union of full sets with the
  same constant is full).
\end{proposition}

\begin{proof}
  Let $A$ and $B$ be full sets with constants $r_A$ and $r_B$ respectively.
  Then it is easy to check that $A \cup B$ is full with constant $r =
  \min\{r_A,r_B\}$. Moreover, let $x \notin A$ and assume
  towards a contradiction that $y \in A$ for some $y \in B(x,r_A)$. By
  the properties of the ultrametric $d_X$, we have that $B(y,r_A) =
  B(x,r_A)$: but since $A$ is full with constant $r_A$, then $B(y,r_A)
  \subseteq A$ and hence $x \in A$, a contradiction! Thus $X \setminus A$ is
  full (with constant $r_A$). The second part follows again from the
  properties of an ultrametric.
\end{proof}

\begin{defin}\label{deffullfuntion}
  Let $(X,d_X)$ be an ultrametric space and $Y$ be any separable
  metrizable space. A function $f : X \to Y$ is said  to be \emph{full} if
  it has only finitely many values and the preimage of each of these
  values is a full set.

The function $f$ is said to be \emph{$\omega$-full} if it has at most
countably many values and there is some fixed $r \in \RR^+$ such that the
preimage of each value is a full set with constant $r$.
\end{defin}

It is clear that every full function is $\omega$-full. Moreover, if
$f$ is $\omega$-full and $r \in \RR^+$ witnesses this, then $f \in
\Lip(X,Y;r^{-1})$ (with respect to any metric $d_Y$ compatible with the
topology of $Y$ such that $d_Y \leq 1$). In fact, let $d_Y$ be such a
metric and let $x,x' \in X$:  if $d_X(x,x') \geq r$ then
\[d_Y(f(x),f(x')) \leq 1 = r^{-1} \cdot r \leq r^{-1} d_X(x,x'),\]
while if $d_X(x,x')<r$ then $x' \in f^{-1}(f(x))$ (since $x' \in
B(x,r)$) and hence $f(x) = f(x')$.

\begin{proposition}\label{propfullLipschitz}
  Let $(X,d_X)$ be an ultrametric space and $(Y,d_Y)$, $(Z,d_Z)$ be
  two  metric spaces. Let $f: X \to Y$ be a full
  function, $g \in \Lip(Y,Z;L)$ and $h \in \Lip(Z,X;L)$. Then $g \circ
  f$ is full and, if $d_Z$ is an ultrametric, also $f \circ h$ is
  full.

The same result holds if we sistematically replace ``full'' with
``$\omega$-full''.
\end{proposition}

\begin{proof}
  The first part is obvious, since for every $z \in Z$ the set $(g
  \circ f)^{-1}(z)$ is either empty or the union of finitely many full
  sets (and the cardinality of ${\rm range}(g \circ f)$ is less or
  equal than the cardinality of ${\rm range}(f)$). For the second
  part, it is enough to show that the preimage via $h$ of a full set $A
  \subseteq X$ (with constant $r$) is a full set (with constant $r \cdot
  L^{-1}$). In fact, let $z \in Z$ be such that $h(z) \in A$ and let $z' \in
  Z$ be such that $d_Z(z,z') < rL^{-1}$. Then $d_X(h(z),h(z')) \leq L
  \cdot d_Z(z,z') < L r L^{-1} = r$, and thus $h(z') \in A$. But this
  implies $B(z,rL^{-1}) \subseteq h^{-1}(A)$ and hence we are done. The case
  in which $f$ is $\omega$-full is proved in a similar way.
\end{proof}

Now we are ready to state the main Theorem of this paper.

\begin{theorem} \label{theorfull}
Let $(X,d_X)$ be an ultrametric space and let $Y$ be any separable metrizable
space. Then $f: X \rightarrow Y$ is of Baire
class $1$ if
and only if $f$ is the pointwise limit of a sequence of full
functions.
\end{theorem}

By the observations above and since every Lipschitz function is
uniformly continuous, we have also
the following
Corollary as a simple
consequence of Theorem \ref{theorfull}.

\begin{corollary}
Let $(X,d_X)$, $(Y,d_Y)$ be separable metric spaces and assume
that $X$ is an ultrametric  space with respect to $d_X$.
For every $f: X \rightarrow Y$ the following are equivalent:
\begin{enumerate}[i)]
\item $f$ is of Baire class $1$;
\item $f$
is the pointwise limit of a sequence of $\omega$-full functions;
\item $f$
is the pointwise limit of a sequence of Lipschitz functions;
\item $f$
is the pointwise limit of a sequence of uniformly continuous
functions.
\end{enumerate}
\end{corollary}

The author first proved Theorem \ref{theorfull} but using Lipschitz
(in particular $\omega$-full) functions rather than full functions
(although the proof was essentially the same presented here in Section
\ref{sectionmain}):
the idea to generalize the result to the present form (as well as the
definition of fullness) is due to S.\ Solecki.

\section{The link between Baire class and Delta functions}\label{sectionlink}

We first give some basic definitions.

\begin{defin}
  Let $X$ be a topological space and $\Gamma \subseteq \pow(X)$ be any
  pointclass. A \emph{$\Gamma$-partition} of a set $C \in \Gamma$ is a
  family $\seq{C_n}{n<N}$ of nonempty pairwise disjoint sets of
  $\Gamma$ such that $C = \bigcup_{n<N} C_n$ and $1 \leq N \leq\omega$
  \footnote{For the rest of the paper we will always assume
without explicitly mentioning it that $N$ is some ordinal less or
equal to $\omega$.}.
\end{defin}

\begin{defin}
Let $X$, $Y$ be two metrizable spaces and let $\F$ be some set of
functions between $X$ and $Y$. Let $f:X \rightarrow Y$ be an
arbitrary function and $\langle C_n \mid n<N \rangle$  be some
partition of $X$.  We say that $f$ is
\emph{(locally) in $\F$ on the partition
  $\langle C_n \mid n<N
\rangle$} if there is a family of functions $\{f_n \mid n<N\} \subseteq
\F$ such that $\restr{f} {C_n}  = \restr{f_n}{C_n}$ for every $n<N$.

Moreover, if $\Gamma \subseteq \pow(X)$ is any pointclass, we will say
that $f$ is \emph{(locally) in $\F$ on a $\Gamma$-partition} if there
is some $\Gamma$-partition such that $f$ is locally in $\F$ on it.
\end{defin}

Obviously, if $\F$ and $\G$ are sets of functions
 and $\F \subseteq \G$, then if $f$ is
locally in $\F$ on the partition $\langle C_n \mid n<N \rangle$ we
have also that $f$ is locally in $\G$ on the same partition.

\begin{proposition} \label{proppartitionDelta0xi}
Let $X$, $Y$ be two metrizable spaces and $\xi$ be some nonzero countable
ordinal. Then every $f:X \rightarrow Y$ is in $\D_\xi(X,Y)$ if and only if
there is a $\bSigma^0_\xi$-partition of $X$ such that $f$ is locally
in $\D_\xi(X,Y)$ on it.
\end{proposition}

\begin{proof}
One direction is trivial, hence we have only to prove that if
$\langle C_n \mid n<N \rangle$ is a $\bSigma^0_\xi(X)$-partition on $X$
and $\{f_n \mid n<N\} \subseteq \D_\xi(X,Y)$ witnesses that $f$ is locally
in $\D_\xi(X,Y)$ on it, then $f \in \D_\xi(X,Y)$. To see this, let $S \in
\bSigma^0_\xi(Y)$.  Then $f^{-1}(S) = \bigcup_{n<N} (f^{-1}(S) \cap C_n) =
\bigcup_{n<N} (f^{-1}_n(S) \cap C_n)$: but $f_n^{-1}(S) \cap C_n \in
\bSigma^0_\xi(X)$ for every $n<N$ and therefore $f^{-1}(S) \in
\bSigma^0_\xi(X)$ (since
$\bSigma^0_\xi(X)$ is  closed under countable unions).
\end{proof}

We are now ready to prove a Theorem from which Theorem
\ref{theorDelta0xi} easily follows. We will use the following
standard fact.

\begin{fact}
Let $(X,d_X)$, $(Y,d_Y)$ be two metric spaces, $\xi$ be  a nonzero
countable ordinal and let $\langle f_n \mid n \in \omega \rangle$
be a sequence of functions from $\B_\xi(X,Y)$ converging uniformly
to some $f:X \rightarrow Y$. Then $f \in \B_\xi(X,Y)$ as well.
\end{fact}

\begin{theorem}\label{theorBaireDelta}
Let $(X,d_X)$, $(Y,d_Y)$ be two metric spaces and assume that $Y$ is also
separable. Then for every function $f:X \rightarrow Y$ the following
conditions are equivalent (for\footnote{If $\xi=0$ then it is
  trivially true that $i) \iff iv) \iff v)$, but $ii)$ and $iii)$ are
  not equivalent to $i)$ unless $X$ is zero-dimensional.} $1\leq \xi
<\omega_1$):
\begin{enumerate}[i)]
\item $f$ is of Baire class $\xi$;
\item there is a sequence of functions $\langle f_k \mid k \in \omega
  \rangle$
  which converges uniformly to $f$ and such that every $f_k$ is
  locally constant on some $\bSigma^0_{\xi+1}$-partition;
\item there is a sequence of functions $\langle f_k \mid k \in \omega
  \rangle$
  which converges uniformly to $f$ and such that every $f_k$ is
  locally Lipschitz on some $\bSigma^0_{\xi+1}$-partition;
\item there is a sequence of functions $\langle f_k \mid k \in \omega
  \rangle$
  which converges uniformly to $f$ and such that every $f_k$ is
  locally continuous on some $\bSigma^0_{\xi+1}$-partition;
\item there is a sequence of functions $\langle f_k \mid k \in \omega
  \rangle$
  which converges uniformly to $f$ and such that every $f_k$ is
  a $\bDelta^0_{\xi+1}$-function.
\end{enumerate}
\end{theorem}

\begin{proof}
It is obvious that $\mathit{ii)} \imp \mathit{iii)}$ and
$\mathit{iii)} \imp \mathit{iv)}$, since every constant function is
Lipschitz and every Lipschitz function is also continuous. Moreover,
using Proposition \ref{proppartitionDelta0xi} and the fact that every
continuous function is $\bDelta^0_\xi$ (for every nonzero
$\xi<\omega_1$) we have also $\mathit{iv)} \imp \mathit{v)}$. Also
$\mathit{v)} \imp \mathit{i)}$ is easy: in fact
 every $\bDelta^0_{\xi+1}$-function is of Baire class $\xi$  and $\B_\xi(X,Y)$ is closed under uniform
limits by the Fact above.\\

Finally we prove $\mathit{i)} \imp \mathit{ii)}$. For every $k \in
\omega$, fix some open cover $\langle U^k_n\mid n < N \rangle$  of
$Y$ of mesh $2^{-k}$, that is a sequence of open sets such that $Y
= \bigcup_{n<N}U^k_n$ and $\diam(U^k_n) \leq 2^{-k}$ for every $n
< N$, and fix also a sequence $\langle z^k_n \mid  n<N \rangle$ of
points of $Y$ such that $z^k_n \in U^k_n$ for every $n<N$. One way
to do this is to consider an enumeration $\langle y_n \mid n<N
\rangle$ of some (at most) countable dense set of points of $Y$ (which
exists because $Y$ is separable) and to put
$U^k_n=B(y_n,2^{-(k+1)}) = \{y \in Y \mid d_Y(y,y_n)<2^{-(k+1)}\}$
and $z^k_n= y_n$. Let $S^k_n = f^{-1}(U^k_n)$. Clearly $\forall
n<N ( S^k_n \in \bSigma^0_{\xi+1}(X))$ (since $f \in \B_\xi(X,Y)$)
and the sets $\langle S^k_n \mid n<N \rangle$ cover $X$. Since
$\bSigma^0_{\xi+1}$ has the generalized reduction property, we can find
for every $k \in \omega$ a sequence $\seq{Q^k_n}{n<N}$ of
$\bSigma^0_{\xi+1}$ sets such that $Q^k_n \subseteq S^k_n$ for every
$n<N$, $Q^k_n \cap Q^k_m =\emptyset$ if $n \neq m$ and $\bigcup_{n<N}
Q^k_n = \bigcup_{n<N}S^k_n =X$ (thus, in particular, the $Q^k_n$'s form a
$\bSigma^0_{\xi+1}$-partition of $X$).

Now define  $f_k:X \rightarrow Y : x \mapsto z^k_n$, where $n<N$ is
the unique natural number such that  $x \in Q^k_n$.

Note that $f_k$ is locally
constant on $\langle Q^k_n \mid  n < N \rangle$. It remains only to
prove that the sequence
$\langle f_k \mid k \in \omega \rangle$ converges uniformly to $f$.
Clearly this follows from
\begin{claim}\label{claimuniform}
Fix some $k \in \omega$. Then for every $x \in X$ \[d(f_k(x),f(x))
\leq 2^{-k}.\]
\end{claim}

\begin{proof}[Proof of the Claim]
Fix some $x \in X$ and let $n$ be such that $x \in Q^k_n$ (so that, in
particular, $x
\in S^k_n$).  Then $f(x)
\in U^k_n$, and since $\diam(U^k_n) \leq 2^{-k}$ and $z^k_n \in
U^k_n$ we have that $d(f_k(x),f(x)) \leq 2^{-k}$.
\renewcommand{\qedsymbol}{$\square$ \textit{Claim}}
\end{proof}
\end{proof}

\begin{remark}
  The same result holds if we consider
$\bDelta^0_{\xi+1}$-partitions because it is easy to check that
every $\bSigma^0_{\xi+1}$-partition is actually a
$\bDelta^0_{\xi+1}$-partition (the converse is trivially
true). We can also consider partitions formed only by sets
which are difference of two $\bPi^0_\xi$ sets (i.e.\
$2\text{-}\bPi^0_\xi$ sets), since every $\bSigma^0_{\xi+1}$-partition
$\seq{S_n}{n<N}$ can be refined to a
$2\text{-}\bPi^0_{\xi}$-partition. In fact, since $S_n \in
\bSigma^0_{\xi+1}$, by definition there are $\langle
P_{n,m} \mid  m  \in
\omega, n<N \rangle$ such that $P_{n,m} \in \bPi^0_\xi$ and
$S_n =\bigcup_{m \in \omega} P_{n,m}$ for every $n<N$ (we are
not requiring that the sets $P_{n,m}$ are different for distinct
indexes $m$, hence we can suppose that $P_{n,m}$ is defined for
every $m\in \omega$).  Fix some
bijection $\langle \cdot, \cdot \rangle$ between $\omega \times
\omega$ and $\omega$ (for example $\langle i,j \rangle =
(2^i(2j+1))-1$) and let $R_{\langle n,m \rangle} =P_{n,m}$.
Inductively put $j_0 = 0$ and $j_{i+1}=\min\{ j \mid  j>j_i \wedge R_j
\setminus \bigcup_{l<j}R_l \neq \emptyset \}$ (in general the
sequence $j_i$ is defined for $i < I$ where $I \leq \omega$). Now
define
\[Q_i =
R_{j_i} \setminus \bigcup_{l<j_i}R_l = R_{j_i} \setminus
\bigcup_{l<i} Q_l\]
for every $i<I$. Clearly $\langle Q_i \mid i<I\rangle$ is an at most
countable partition
of $X$ (since the sets $S_n$ cover $X$) and refines
$\seq{S_n}{n<N}$. Moreover every $Q_i$ is the
difference of two $\bPi^0_\xi$
sets (being $\bPi^0_\xi$ is closed under finite unions).

Moreover, if $d_Y$ is a compact metric (e.g.\ it is  induced by
any metric on a compactification of $Y$),  the partitions above can
be taken to be finite.

Finally, as we will see in the next Section, if $X$ is
zero-dimensional we can strengthen the result a little bit by
taking $\bPi^0_\xi$-partitions (instead of
$\bSigma^0_{\xi+1}$-partitions) in conditions \textit{ii)-v)}.

\end{remark}

Now we restate Corollary \ref{corDelta0xi} in the following
(slightly) stronger form.

\begin{proposition}
Let $X$, $Y$ be two metrizable spaces and assume that $Y$ is also
separable. Then for every nonzero $\xi<\omega_1$, if there is a
sequence of $\bDelta^0_{\xi}$-functions pointwise converging to
$f$ then $f$ is of Baire class $\xi$.

Conversely, if $\xi>1$ and $f:X \rightarrow Y$ is of Baire class
$\xi$ then there is a sequence of $\bDelta^0_{\xi}$-functions
pointwise converging to $f$.

\end{proposition}

\begin{proof}
Let $d$ be a compatible metric on $Y$. Recall that every open
sphere $U$ of $Y$ can be written as the union of countably many
closed sphere each of which is contained in the interior of the
following one. In fact let $U = B(y_0,\varepsilon) = \{ y \in Y
\mid d(y,y_0) < \varepsilon \}$ and let $\langle \varepsilon_m
\mid m \in \omega \rangle$ be a strictly increasing sequence of real such
that $\varepsilon_m < \varepsilon$ for every $m \in \omega$ and
$\lim_m \varepsilon_m = \varepsilon$: then $U = \bigcup_{m \in
\omega} \overline{B_m(y_0,\varepsilon_m)} = \bigcup_{m \in \omega}
\{ y \in Y \mid d(y,y_0) \leq \varepsilon_m\}$. Moreover, since
$\varepsilon_n < \varepsilon_{n+1}$, we have also $\overline{B_m}
= \overline{B(y_0,\varepsilon_m)} \subseteq B_{m+1} =
B(y_0,\varepsilon_{m+1})$, that is $U = \bigcup_{m \in \omega}
B_m$.

Now assume that $\langle f_k \mid k \in \omega \rangle$ is a
sequence of $\bDelta^0_\xi$-functions pointwise converging to $f$:
it is enough to prove that $f^{-1}(U) \in \bSigma^0_{\xi+1}(X)$
for every open sphere $U = \bigcup_m \overline{B_m} \subseteq Y$.
First note that
\[ f^{-1}(U) = \bigcup_{m \in \omega} \bigcup_{n \in \omega}
\bigcap_{k\geq n} f^{-1}_k(\overline{B_m}).\]
In fact, if $f(x)
\in U$ then there is an $m$ such that $f(x) \in B_m \subseteq \overline{B_m}$
and hence
also $f_k(x) \in B_m$ for any $k$ large enough (since $f_k$ converge
to $f$). For the other direction, if there is some $m$ such that
$f_k(x) \in \overline{B_m}$ for almost all $k$, thus also $f(x)$
(which is the limit of the points $f_k(x)$) must belong to the
same $\overline{B_m}$ (since it is closed).

Since each $f_k$ is a $\bDelta^0_\xi$-function and since
$\bPi^0_1(Y) \subseteq \bPi^0_\xi(Y)$ for every nonzero countable
$\xi$, we have that
$f^{-1}(\overline{B_m}) \in \bPi^0_\xi(X)$ for every $m \in \omega$
and hence also
$\bigcap_{k \geq n} f^{-1}_k(\overline{B_m}) \in \bPi^0_\xi(X)$
for every $n \in \omega$ (since $\bPi^0_\xi(X)$ is closed under
countable intersections). But then $f^{-1}(U)$ is a countable union of
$\bPi^0_\xi(X)$ sets, i.e.\ it is a $\bSigma^0_{\xi+1}(X)$ set and
we are done.\\

Conversely, if $\xi>1$ and $f$ is of Baire class $\xi$ then it is
the pointwise limit of some sequence $f_n$ of functions such that
for every $n \in \omega$ there is a $1\leq \nu_n < \xi$ such that $f_n$
is of Baire class $\nu_n$. Using Theorem \ref{theorDelta0xi}, find
for every $n \in \omega$ a sequence $g_{n,m}$ of
$\bDelta^0_{\nu_n+1}$-functions converging uniformly to $f_n$.
Note that by the construction above (Claim \ref{claimuniform}) we
can assume that $d(g_{n,m}(x),f_n(x)) \leq
2^{-m}$ for every $x \in X$. Moreover, since $\nu_n+1 \leq \xi$ we
have that every
$g_{n,m}$ is, in particular, a $\bDelta^0_{\xi}$-function. Take
any diagonal subsequence $\langle h_n \mid n \in \omega \rangle$ of
the $g_{n,m}$, e.g.\ $h_n = g_{n,n}$. It remains only to prove
that this sequence converges pointwise to $f$. To see this, fix
some $x \in X$ and $k \in \omega$. Let $j \in \omega$ be such that
\[ \forall i \geq j \; (d(f_i(x),f(x)) < 2^{-(k+1)})\] and put $m =
\max\{j,k+1\}$. Clearly, for every $m' \geq m$ we have
\begin{multline*}
d(h_{m'}(x),f(x))
\leq
d(g_{m',m'}(x),f_{m'}(x))+d(f_{m'}(x),f(x)) < \\
<2^{-m'} + 2^{-(k+1)} \leq 2 \cdot 2^{-(k+1)} = 2^{-k}.
\end{multline*}
\end{proof}

The same Corollary clearly holds if we consider functions which are
constant (respectively, Lipschitz,  continuous) on a \emph{finite}
$\bDelta^0_\xi$-partition.

\section{Zero dimensional spaces}\label{sectionzero}

We now prove some  Theorems on zero-dimensional and ultrametric
spaces. In particular, the first is a simple variation of some
classical results (see \cite{kechris}).
Let $s \in {}^{<\omega}\omega$ be a finite sequence of natural
numbers. We will denote the \emph{length of $s$} by $\leng(s)$
(formally, $\leng(s) = \dom(s)$).

\begin{theorem}\label{theorpolishspaces}
If $(X,d)$ is a metric, separable  and  zero-dimensional space,
then there is some set $A\subseteq \Bai$ and an
homeomorphism $h:A \rightarrow X$ such that $h \in \Lip(A,X;1)$
(with respect to $d$ and the usual metric $d'$ that $\Bai$ induces
on $A$).

If moreover $d$ is an ultrametric then $h$ can be taken
bi-Lipschitz, i.e.\ $h^{-1} \in \Lip(X,A;2)$ (and $h \in
\Lip(A,X;1)$ as before).

If $d$ is also complete then the set $A$ can be taken to be a
closed set.
\end{theorem}

\begin{proof}
The first part is a standard argument: one can construct a Lusin
scheme $\langle C_s \mid s \in {}^{<\omega}\omega\rangle$ on $X$ such that
\begin{enumerate}[i)]
\item $C_\emptyset = X$
\item $C_s$ is clopen
\item $C_s = \bigcup_{i \in \omega} C_{s \conc i}$
\item $\diam(C_s) \leq 2^{-\leng(s)}$.
\end{enumerate}
From this one can conclude that the induced map $f$ is defined on
 the set $A = \{y \in \Bai \mid \bigcap_n C_{\restr{y}{n}} \neq
 \emptyset \}$ and is an homeomorphism. But condition
\textit{iv)} implies also $f \in \Lip(A,X;1)$. In fact, for every
$x,y \in A$ such that $x \neq y$, let $n \in \omega$ be such that
$d'(x,y) = 2^{-n}$ and let $s = \restr{x}{n} = \restr{y}
{n}$. Clearly we have that $h(x) \in C_s$ and $h(y) \in
C_s$. Thus condition \emph{iv)} implies that $d(h(x),h(y)) \leq
2^{-\leng(s)} = 2^{-n} = d'(x,y)$.

If we now assume that $d$ is an ultrametric on $X$ then we can
construct a Lusin scheme $\langle C_s \mid s \in
{}^{<\omega}\omega\rangle$ on $X$
such that
\begin{enumerate}[i)]
\item $C_\emptyset = X$
\item either $C_s =\emptyset$ or $C_s$ is a sphere
\item $C_s = \bigcup_{i \in \omega} C_{s \conc i}$
\item $\diam(C_s) \leq 2^{-\leng(s)}$.
\end{enumerate}

In fact every nonempty $C_s$ (with $s \neq \emptyset$) will be defined
as $C_s =
B\left(x,2^{-\leng(s)} \right)$ for some $x \in X$.

Let $D$ be countable and dense in $X$: we construct the scheme by
induction on $\leng(s)$. First put $C_\emptyset = X$. Suppose to have
constructed $C_s$ with properties \textit{i)-iv)}. If $C_s =
\emptyset$ then put $C_{s\conc i}= \emptyset$ for every $i \in
\omega$, otherwise fix an enumeration $\langle x_k \mid k \in \omega
\rangle$ of $C_s \cap D$. Then define $C_{s \conc 0} = B \left(
  x_0,2^{-(\leng(s)+1)} \right)$ and either $C_{s \conc i+1} = B
\left( x_{k_{i+1}},2^{-(\leng(s)+1)}\right)$, where
$k_{i+1}$ is the smallest $k> k_i$ such that $x_k \notin \bigcup_{j
  \leq i}C_{s \conc j}$, or  $C_{s \conc i+1} =
\emptyset$ if such a $k$ does not exist.

Clearly $C_{s \conc i} \subseteq C_s$ (since $d$ is an ultrametric),
$\diam(C_{s \conc i}) \leq 2^{-(\leng(s)+1)}$
and $C_s = \bigcup_{i \in \omega} C_{s \conc i}$ because $D$ is dense,
hence we are done. Arguing as before, $h$ is a bijection defined on a
 set $A \subseteq \Bai$ and $h \in \Lip(A,X;1)$. Now we want to
show that  $d'(h^{-1}(x),h^{-1}(y)) \leq 2
d(x,y)$ for every distinct $x,y \in X$. Put $S_{x,y} = \{s
\in{}^{<\omega}\omega \mid  x \in C_s
\wedge y \in C_s\}$. Clearly $S_{x,y}$ is linearly ordered and admits
an element $t$ of maximal length (otherwise $x=y$). Thus
$d'(h^{-1}(x),h^{-1}(y))= 2^{-\leng(t)}$. If $d(x,y) <
2^{-(\leng(t)+1)}$ then, by the construction above and the fact
that $d$ is an ultrametric, there would be an $i \in \omega$ such
that $x \in C_{t \conc i}$ and $y \in C_{t \conc i}$, contradicting
the maximality of $t$. Hence $d(x,y) \geq 2^{-(\leng(t)+1)}$
and
\[d'(h^{-1}(x),h^{-1}(y)) = 2^{-\leng(t)} = 2\cdot
2^{-(\leng(t)+1)} \leq 2 d(x,y) \]
as required.

Finally it is not hard to check that the completeness of $d$ implies
that $A$ is a closed set.
\end{proof}

Note that, in particular, this Theorem provides also that every
separable, metrizable and zero-dimensional space is ultrametrizable
(i.e.\ it admits a compatible ultrametric $d$): let $h$ be the
homeomorphism given by the Theorem and simply put $d(x,x') =
d'(h^{-1}(x),h^{-1}(x'))$ for every $x,x' \in X$, where $d'$ is the
  usual (ultra)metric on $\Bai$. Clearly, if $X$ is also Polish we
   have that the ultrametric $d$ is also complete.

For notational simplicity we put $\bSigma^0_0 = \bPi^0_0 = \bDelta^0_0
= \bDelta^0_1$. Moreover, for every countable ordinal $\xi$, we denote
by $\bPi^0_{<\xi}$ (respectively, $\bSigma^0_{<\xi}$ and
$\bDelta^0_{<\xi}$) the pointclass $\bigcup_{\nu<\xi}\bPi^0_\nu$
(resp.\ $\bigcup_{\nu<\xi}\bSigma^0_\nu$ and
$\bigcup_{\nu<\xi}\bDelta^0_\nu$).

\begin{theorem}
Let $X$ be a separable, metrizable, zero-dimensional space and let
$A$ be a subset of $X$. For every nonzero $\xi < \omega_1$ the
following are equivalent:
\begin{enumerate}[i)]
\item $A \in \bSigma^0_\xi$;
\item there is a $\bDelta^0_\xi$-partition of $A$, i.e.\ there is
  $\langle C_n \mid n < N \rangle$  such that for every $n,m < N$ we have
  $C_n \in \bDelta^0_\xi$, $n \neq m \imp C_n \cap C_m =
  \emptyset$, and $A = \bigcup_{n < N} C_n$;
\item there is a $\bPi^0_{<\xi}$-partition of $A$,
  i.e.\ there is $\langle P_n \mid n<N \rangle$ such that for every $n<N$ there is some $\nu_n < \xi$
  with $P_n \in \bPi^0_{\nu_n}$, if $n \neq m$ then $P_n \cap P_m =
  \emptyset$, and $A =\bigcup_{n<N} P_n$.
\end{enumerate}
\end{theorem}

\begin{proof}
The implication $\mathit{iii)} \imp \mathit{ii)}$ is obvious since
every $\bPi^0_\nu$ set is also $\bDelta^0_\xi$ if $\nu < \xi$. Also
$\mathit{ii)} \imp \mathit{i)}$ is easy since every $\bDelta^0_\xi$
set is by definition a $\bSigma^0_\xi$ set and the latter pointclass
is closed under countable unions. Hence we have only to prove
$\mathit{i)} \imp \mathit{iii)}$ and this will be done by induction on $1
  \leq \xi < \omega_1$.

If $\xi = 1$ we have only to note that every open set $U$ can be
written as a countable union of pairwise disjoint clopen sets.
Since $X$ is separable and zero-dimensional we have that $U =
\bigcup_n C_n$ for some sets $C_n \in \bDelta^0_1$: now define
by induction $P_0 = C_0$ and $P_{n+1} = C_{n+1} \setminus
\bigcup_{i\leq n} C_i$ and note that each $P_n$ is clopen (since
$\bDelta^0_1$ is closed under complementation and finite unions
and intersections), $U = \bigcup_n P_n$ and that $n \neq m \imp
P_n \cap P_m = \emptyset$.

If $\xi > 1$ and $S \in \bSigma^0_\xi$, by definition there are
some sets $P_n \in \bPi^0_{\nu_n}$ such that $S = \bigcup_n P_n$
and $\nu_n < \xi$ for all $n \in \omega$. First define inductively
$P'_0 = P_0$ and $P'_{n+1} = P_{n+1} \setminus \bigcup_{i \leq n}
P_i$ and note that they form a partition of $S$. Clearly each
$P'_n$ can be seen as the difference of two $\bPi^0_\nu$ sets
where $\nu = \max\{ \nu_0, \dotsc, \nu_n\} < \xi$  (since
$\nu' \leq \nu \imp \bPi^0_{\nu'}\subseteq \bPi^0_\nu$ and
$\bPi^0_\nu$ is closed under finite unions) and hence we have only
to prove that for all $\nu < \xi$, every set of the form $Q \cap
R$ with $Q \in \bPi^0_\nu$ and $R \in \bSigma^0_\nu$ admits a
$\bPi^0_\nu$ partition. Using the inductive hypothesis, find a
partition $\langle R_n \mid n \in \omega \rangle$ of $R$ such that
$R_n \in \bPi^0_{\mu_n}$ for some $\mu_n < \nu$ and note that  $R_n
\in \bPi^0_\nu$ for
every $n \in \omega$. Then it is easy to
check that the sets $Q_n = Q \cap R_n$ are in $\bPi^0_\nu$ and
that they form a partition of $Q \cap R$, hence we are done.
\end{proof}

In particular, every $\bSigma^0_{\xi+1}$ set admits a
$\bPi^0_\xi$-partition. This implies that every
$\bSigma^0_{\xi+1}$-partition of $X$ can be refined to a
$\bPi^0_\xi$-partition and, more generally, every
$\bSigma^0_\xi$-partition can be refined to a
$\bPi^0_{<\xi}$-partition. Therefore, in the case $X$ is separable,
metrizable and zero-dimensional, we have the following improvement
of Theorem \ref{theorBaireDelta}.

\begin{corollary}
Let $(X,d_X)$ and $(Y,d_Y)$ be two metric separable spaces and
assume that $X$ is also zero-dimensional. Then a function $f: X
\rightarrow Y$ is of Baire class $\xi$ if and only if there is a
sequence of functions converging uniformly to it and such that
each of them is locally constant (respectively, Lipschitz,
continuous) on a $\bPi^0_\xi$-partition of $X$.
\end{corollary}

\section{Baire class $1$ and full functions}\label{sectionmain}

Let $\bGamma \subseteq \mathscr{P}(\Bai)$ be a boldface
pointclass, i.e.\ a collection of subsetes of $\Bai$ closed
under continuous preimage. We say that a set $A \in \bGamma$ is
\emph{$\bGamma$-complete} if for every $B \in \bGamma$ there is a
continuous function $f:\Bai \rightarrow \Bai$ such that $B =
f^{-1}(A)$ (such a function will be called a \emph{reduction of $B$
in $A$}).

Recall also that a continuous function from $\Bai$ to $\Bai$ can
be viewed as the function arising from some particular function
$\fhi:{}^{<\omega}\omega \rightarrow {}^{<\omega}\omega$.
We say that  $\fhi:\seqo \rightarrow {}^{<\omega}\omega$ is
\emph{continuous} if  $s \subseteq t \imp
\fhi(s) \subseteq \fhi(t)$ for every $s,t \in \seqo$ and for every $x \in \Bai$
\[\lim_{n \in
\omega} (\leng(\fhi(\restr{x}{n}))) =\infty.\]
If $\fhi$ is
continuous it induces in a canonical way the unique function
\[f_\fhi:\Bai \rightarrow \Bai : x \mapsto \bigcup_{n \in \omega}
\fhi(\restr{x}{n}),\]
and it is not hard to see that $f_\fhi$ is a continuous function.

Conversely, suppose $f:\Bai \rightarrow \Bai$ is continuous. For every
$s \in \seqo$ consider the set $\Sigma_s = \{ t \in
{}^{<\omega}\omega\mid  f(\bN_s) \subseteq \bN_t\}$. Clearly
$\Sigma_s$ is linearly ordered (because if $t$ and $t'$ are
incompatible then $\bN_t \cap \bN_{t'}=\emptyset$), and hence we can
define $\fhi(s) = t_s$ where $t_s \in \Sigma_s$ is such that
$\leng(t_s) = \max\{\leng(t) \mid  \leng(t) \leq \leng(s) \wedge t \in
  \Sigma_s\}$. It is not difficult to check that
  $\fhi:\seqo \rightarrow
  {}^{<\omega}\omega$ is continuous and that $f_\fhi  = f$.

By analogy with the previous definitions,
if $A,B \subseteq \Bai$ and $\fhi:{}^{<\omega}\omega
\rightarrow {}^{<\omega}\omega$ is a continuous function such that
$f_\fhi^{-1}(A) = B$, we call $\fhi$ a \emph{reduction of $B$ into
  $A$} and we say that $\fhi$ \emph{reduces} $B$ to $A$. From the
observation above, it is clear that if $A$ is $\bGamma$-complete
for some pointclass $\bGamma \subseteq \mathscr{P}(\Bai)$ then for
every $B \in \bGamma$ there is a reduction $\fhi:{}^{<\omega}
\omega \rightarrow {}^{<\omega}\omega$ of $B$ in $A$.\\

For every $t,s \in {}^{<\omega}\omega$ define $t-s = \emptyset$ if
$\leng(t) < \leng(s)$, and $t-s = u \in {}^{<\omega}\omega$, where $u$ is such
that $t = (\restr{t}{\leng(s)}) \conc u$, otherwise.

Let $\vec{\fhi} = \langle \fhi_n \mid n<N  \rangle$  be a sequence
of continuous functions $\fhi_n: \seqo \rightarrow {}^{<\omega}\omega$.
Moreover, let
$\langle n_k \mid k \in \omega \rangle$ be an enumeration of $N$
with infinite repetitions such that\footnote{Clearly this last condition
is required only if $N>1$.} $n_k \neq n_{k+1}$ for every $k
\in \omega$. Define $(\vec{\fhi})^* : \seqo \rightarrow
{}^{<\omega}\omega$ and $\sigma: \seqo \rightarrow N$ in the following
way: first put $(\vec{\fhi})^*(\emptyset) = \emptyset$ and
$\sigma(\emptyset) = n_0$. Then suppose to have defined
$(\vec{\fhi})^*(s)$ and $\sigma(s) = n_k$ and inductively put
\[(\vec{\fhi})^*(s \conc i) = (\vec{\fhi})^*(s) \conc 1\] if
$\fhi_{\sigma(s)}(s \conc i ) - (\vec{\fhi})^*(s)$ does not
contain $0$, and \[(\vec{\fhi})^*(s \conc i) = (\vec{\fhi})^*(s)
\conc 0\] otherwise. Finally put $\sigma(s \conc i) = \sigma(s) =
n_k$ in the first case and $\sigma(s \conc i) = n_{k+1}$ in the second
one.

The function $(\vec{\fhi})^*$ is clearly continuous (since it is
constructed extending at each step the previous value and is such
that $\leng((\vec{\fhi})^*(s)) = \leng(s)$ for every $s \in {}^{<
\omega}\omega$) and is called the \emph{$\bSigma^0_2$-control
function}\footnote{The symbol $\bSigma^0_2$ refers to the
$\bSigma^0_2$-sets which are involved in Claim \ref{claimmixing} and
in the other considerations below.} of the sequence $\vec{\fhi}$, while
the function $\sigma$
is the \emph{state function} associated to it. Moreover we will
say that $\sigma(s) \in N$ is the
\emph{state of $s$ with respect to $(\vec{\fhi})^*$}.\\

Consider now a family $A_n \subseteq \Bai$ of $\bSigma^0_2$ sets
(for $n<N$) and $S = \{ x \in \Bai\mid  \exists n \,\forall m
\geq n (x(m) \neq 0) \}$. Since $S$ is
$\bSigma^0_2$-complete  there are continuous functions $\fhi_n :
{}^{<\omega}\omega \rightarrow {}^{<\omega}\omega$ which reduce
$A_n$ to $S$, i.e.\ such that $f_{\fhi_n}^{-1}(S) = A_n$.
Define $\vec{\fhi} = \langle \fhi_n \mid n <N \rangle$ and let
$(\vec{\fhi})^*$ and
$\sigma$ be constructed as above. For notational simplicity we put
$\phi = (\vec{\fhi})^*$. We want to prove the following

\begin{claim}\label{claimmixing}
The function $f_\phi : \Bai \rightarrow \Bai$ is a reduction of
$\bigcup_{n < N} A_n$ in $S$, i.e.\ $f_\phi^{-1}(S) =
\bigcup_{n<N}A_n$ (hence, in particular, $\bSigma^0_2$ is
closed under countable unions). Moreover, $x \in \bigcup_{n<N}
A_n$ if and only if the sequence $\langle \sigma(\restr{x}{k})
\mid k \in \omega \rangle$ is eventually constant, that is
$\exists m \, \forall m' \geq m (\sigma(\restr{x} {m'}) =
\sigma(\restr{x}{m}))$.
\end{claim}

\begin{proof}[Proof of the Claim]
First observe that, by the definition of $\phi$, for every $x \in
\Bai$ we have $f_\phi(x) \in S$ if and only if the sequence $\langle
\sigma(\restr{x}{m})\mid m \in \omega \rangle$ is eventually
constant, since for every $s \in \seqo$ and $i \in \omega$
 we have that $\phi(s \conc i) = \phi(s) \conc 0$
if and only if $\sigma(s \conc i) \neq \sigma(s)$. So it is enough
to prove that $x \in \bigcup_{n<N} A_n \iff \langle
\sigma(\restr{x} {m})\mid m \in \omega \rangle$ is eventually constant.

For every $k \in \omega$ let $o(n_k) = | \{ i \leq k\mid  n_i =
n_k\}|$. Now suppose that $x \in A_n$ for some $n<N$ and let $l =
|\{n \in \omega \mid  f_{\fhi_n}(x)(n) = 0 \}|$. Since $\fhi_n$ is
a reduction of $A_n$ in $S$ we have that $l<\omega$. Let $k$ be
such that $n_k = n$ and $o(n_k) = l+1$. If there is no $m$ such
that $\sigma(\restr{x}{m}) = n_k$ then the sequence of states of
$\restr{x}{i}$ (for $i \in \omega$) with respect to $\phi$ is eventually
constant and we are
done. Otherwise, there are $0<m_0 < m_1< \dotsc < m_{l-1} < m$ such that
$\sigma(\restr{x}{(m_i-1)})=n$ and $\phi(\restr{x}{m_i}) -
\phi(\restr{x}{(m_i-1)})$ contains a $0$ for every $i < l$.
Therefore for every $m' \geq m$ we have that
$\fhi_{n_k}(\restr{x}{(m'+1)}) - \phi(\restr{x} {m'})$ does not
contain any $0$ since $\fhi_{n_k} = \fhi_n$: hence
$\sigma(\restr{x} {m'}) = n_k$ for every $m' \geq m$ and we are
done again.

For the other direction, assume $x \notin \bigcup_{n<N}A_n$. Then
for every $n<N$ and $m \in \omega$ we have that there is an $m'
> m$ such that $\fhi_n(\restr{x}{m'}) - \fhi_n(\restr{x}{m})$ contains
some $0$. This implies that for every
$m \in \omega$ there is an $m'>m$ such that $\sigma(\restr{x}
{m'}) \neq \sigma(\restr{x}{m})$ and thus
$\langle \sigma(\restr{x}{m})\mid m \in \omega \rangle$ is
not eventually constant.
\renewcommand{\qedsymbol}{$\square$ \textit{Claim}}
\end{proof}

With the notation above, if $x \in \bigcup_{n<N} A_n$, we will call
\emph{stabilizing point of $\phi$ on $x$} the natural number

\[m_{x,\phi} = \min \{ m\in \omega\mid \forall m' \geq m
(\sigma(\restr{x}{m'}) = \sigma(\restr{x}{m}))\}.\]
Moreover it is not
difficult to check that in this case $x \in A_{\sigma(t)}$ (where
$t=\restr{x}{m_{x,\phi}}$). In fact, $\fhi_{\sigma(t)}(\restr{x}
{m'})-\phi(\restr{x}{(m'-1)})$ does not contain any $0$ for every
$m' \geq m_{x,\phi}$,  and hence
$f_{\fhi_{\sigma(t)}}(x) \in S$.\\

Now we state and prove a Theorem which is crucial to obtain
Theorem \ref{theorfull}.

\begin{theorem} \label{theorcarLipschitz}
Let $A$ be any subset of $\Bai$ and
$Y$ be a separable metrizable space. If $f\in \B_1(A,Y)$ then
there is a sequence of full functions $f_k: A \rightarrow
Y$ converging pointwise to $f$.
\end{theorem}

\begin{proof}
Let $d_Y$ be any compatible metric on $Y$.
Let $\langle U_s\mid s \in {}^{<\omega}\omega \rangle$ be an \emph{open
  scheme on $Y$},
i.e.\ a family of sets $U_s \subseteq Y$ such that for every $s \in
{}^{<\omega}\omega$ we have\footnote{Note that in general this is not a Lusin
  scheme since, in this case, we do not require that if $s$ and $t$ are
  incompatible sequences then $U_s \cap U_t = \emptyset$. However we
  can add this condition if $Y$ is also zero-dimensional.}:
\begin{enumerate}[i)]
\item $U_\emptyset= Y$
\item $U_s$ is open
\item $U_s = \bigcup_{i \in \omega} U_{s \conc i}$
\item $\diam(U_s) \leq 2^{-\leng(s)}$.
\end{enumerate}

One way to do this is to fix some countable dense $D \subseteq Y$
(which exists since $Y$ is separable) and an enumeration\footnote{We
allow repetitions if $Y$ is finite.} $\langle
y_i\mid i \in \omega \rangle$ of it, and then recursively define
$U_\emptyset = Y$ and $U_{s \conc i} = B(y_i,s^{-(\leng(s)+2)})
\cap U_s$. Note that $U_s$ could be the empty set for some $s$,
but the sequences $s$ such that $U_s \neq \emptyset$ form a pruned
tree $R$ on $\omega$: hence for every $s \in R$ we can fix some
$y_s \in U_s$.

Since $f \in \B_1(A,Y)$ and $f^{-1}(U_s) \in \bSigma^0_2(A)$, there
are $V_s \in \bSigma^0_2(\Bai)$ such that $f^{-1}(U_s) = V_s \cap A$
(for every $s \in R$), thus we can consider some reduction
$\tilde{\fhi}_s:\seqo \rightarrow {}^{<\omega}\omega$ of $V_s$
in $S = \{ x \in \Bai\mid \exists n\, \forall m \geq n (x(m) \neq
0)\} \subseteq \Bai$. Moreover, for all these
$s$ we can consider an enumeration without repetitions $\langle
j_i\mid  i < I_s \rangle$ ($I_s \leq \omega$) of the $j \in
\omega$ such that $s \conc j \in R$, and define the sequence of
continuous functions $\vec{\fhi}_s = \langle \fhi_i\mid i < I_s
\rangle$ where $\fhi_i= \tilde{\fhi}_{s \conc
  j_i}$. Finally, let $\psi_s = (\vec{\fhi}_s)^*$ be the
$\bSigma^0_2$-control function of the sequence $\vec{\fhi}_s$, and
$\sigma_s$ be the state function associated to it.

We are now ready to define the functions $f_k$. Fix some $k \in
\omega$ and for every $x \in A$ inductively define for $i < k$:
\[ \begin{cases}
s^{x,k}_0 = \langle \min \{\sigma_\emptyset(\restr{x}{k}),k\} \rangle &\\
s^{x,k}_{i+1} = t_i \conc \min\{\sigma_{t_i}(\restr{x}{k}),k\},&
\end{cases} \]
where $t_i = s^{x,k}_i$.
For notational simplicity, for every $n \in \omega$ we will put
$s^x_n = s^{x,n}_n$. Note that by definition of $\vec{\fhi}_s$ one
can easily prove by induction on $i \leq k$ that $s^x_k \in R$ for
every $x \in A$: hence we can define
\[ f_k : A \rightarrow Y : x \mapsto y_{s^x_k}.\]

\begin{claim}
For every $k \in \omega$ the function $f_k$ is full with constant $k$.
\end{claim}

\begin{proof}[Proof of the Claim]
It is clear that $s^x_k \in {}^{k+1}(k+1)$ for every $x \in A$, thus $f_k$
has at most $(k+1)^{k+1}$ values. Moreover, these values  depend
only on the sequence $\restr{x}{k}$, hence the preimage of each of them
 is a union of balls with radius $2^{-k}$ (and hence is a full
set by Proposition \ref{propfullbasicproperties}).
\renewcommand{\qedsymbol}{$\square$ \textit{Claim}}
\end{proof}

\begin{claim}
The sequence $\langle f_k\mid k \in \omega \rangle$ converges to
$f$ pointwise.
\end{claim}

\begin{proof}[Proof of the Claim]
Fix some $x \in A$ and $n \in \omega$. We want to prove that
there is an $m \in \omega$ such that $\forall m' \geq m \,(
d_Y(f_{m'}(x),f(x)) \leq 2^{-(n+1)})$.

We define inductively a sequence $\langle m_j\mid  j \leq n
\rangle$ of natural numbers (the sequence of the stabilizing points of
$x$) and a sequence $\langle t_j \mid j \leq n \rangle$ of compatible
and length increasing sequences of natural numbers. First
put $m_0 = m_{x,\psi_\emptyset}$ and $t_0 = \langle \sigma_\emptyset
(\restr{x}{m_0}) \rangle$: then for every $i<n$ define $m_{i+1}
= m_{x,\psi_{t_i}}$ and $t_{i+1} = t_i \conc
\sigma_{t_i}(\restr{x}{m_{i+1}})$. Finally  put $t_{-1} =\emptyset$
by definition.

Now recall that, by Claim \ref{claimmixing} and the observations
following it, if $f(x) \in \bigcup_{m \in \omega} U_{s\conc m}$ then
$f(x) \in U_{s \conc \sigma_s(\restr{x}{m})}$ for every $m \geq
m_{x,\psi_s}$. Therefore the fact that $f(x) \in Y$ implies that $f(x) \in
U_{t_0}$. Moreover, using the same argument, one can show that since
$f(x) \in U_{t_i}$ then $f(x) \in U_{t_{i+1}}$ for every $i<n$, hence
we have $f(x)
\in U_{t_n}$.

Recall also that, by definition of the numbers $m_i$,
\[ \forall m' \geq
m_i \;( t_{i} = t_{i-1} \conc \sigma_{t_{i-1}}(\restr{x}{m'})).\]
Let $m = \max\{m_0, \dotsc, m_n,n,k \}$, where $k$ is the smallest natural
number such that $t_n \in {}^{<\omega} k$. Again  by
induction  on $i
\leq n$, it is not hard to prove that for every $m' \geq m$ and
every $i \leq n$ we have $s^{x,m'}_i = t_i$ and hence $s^x_{m'}
\supseteq t_n$. Since we have that $f_{m'}(x)=y_{s^x_{m'}}
\in U_{s^x_{m'}} \subseteq U_{t_n}$ (by $s^x_{m'} \supseteq t_n$), $\leng(t_n)
= n+1$, and
$\diam(U_{t_n}) = 2^{-\leng(t_n)}$, we can conclude that
$d_Y(f_{m'}(x),f(x)) \leq 2^{-(n+1)}$ and we are done.
\renewcommand{\qedsymbol}{$\square$ \textit{Claim}}
\end{proof}
\end{proof}

We are now ready to prove the characterization of the Baire class
$1$ functions as pointwise limits of full functions, i.e.\
Theorem \ref{theorfull}.
\newline

\begin{proof}[Proof of Theorem \ref{theorfull}]
Since every
full function is Lipschitz and every Lipschitz function is continuous,
if $f$ is the pointwise  limit of
a sequence of full functions then it is in $\B_1(X,Y)$ (see note
\fullref{note}).\\

For the other direction, let $X$ and $Y$ be as in the hypotheses
of the Theorem, and let $A \subseteq \Bai$ and $h:A \rightarrow X$ be
obtained applying the second part of Theorem
\ref{theorpolishspaces} to $X$. Define
\[g = f \circ h : A \to Y.\]
Since  $h$ is continuous, if $f$ is of Baire class $1$ then also
$g$ is of Baire class $1$. Let $g_n : A \rightarrow Y$ be the
sequence of full functions convergent (pointwise) to $g$ that
comes from Theorem \ref{theorcarLipschitz}, and define for every $n
\in \omega$
\[ f_n = g_n \circ h^{-1} : X \to Y.\]

Clearly each $f_n$ is a full function by Proposition
\ref{propfullLipschitz}, and moreover $f$ is the pointwise limit of
the sequence
$\langle f_n\mid n \in \omega \rangle$: in fact for every $x \in
X$ and every $n \in \omega$ we have that $d(f_m(x),f(x)) =
d(g_m(h^{-1}(x)),g(h^{-1}(x))) \leq 2^{-n}$ for $m$ large enough
(since $g_m \rightarrow g$ pointwise). This completes the proof.
\end{proof}

\end{document}